\documentclass[12pt]{amsart}
\usepackage{amssymb, hyperref, mathrsfs,xcolor, enumerate}

\usepackage[margin=1.3in]{geometry}

\newtheorem{theoremIntro}{Theorem}[]

\newtheorem{questionIntro}{Question}

\newtheorem{theorem}{Theorem}[section]
\newtheorem{lemma}[theorem]{Lemma}

\newtheorem{definition}[theorem]{Definition}
\newtheorem{corollary}[theorem]{Corollary}

\theoremstyle{remark}
\newtheorem{remark}[theorem]{Remark}

\numberwithin{equation}{section}

\newcommand{\calF}{\ensuremath{\mathcal{F}}}

\newcommand{\scrW}{\ensuremath{\mathscr{W}}}

\newcommand{\mo}{{-1}}
\newcommand{\modu}{\ensuremath{\mathrm{\;mod\;}}}

\newcommand{\bbZ}{\ensuremath{\mathbb{Z}}}
\newcommand{\bbN}{\ensuremath{\mathbb{N}}}

\begin{document}
\title{Asymptotic complements in the integers}
\author{Arindam Biswas}
\address{Universit\"at Wien, Fakult\"at f\"ur Mathematik, Oskar-Morgenstern-Platz 1, 1090 Wien, Austria \& Erwin Schr\"odinger International Institute for Mathematics and Physics (E.S.I.) Boltzmanngasse 9, 1090 Wien, Austria}
\curraddr{}
\email{arindam.biswas@univie.ac.at\\
		arin.math@gmail.com}
\thanks{}

\author{Jyoti Prakash Saha}
\address{Department of Mathematics, Indian Institute of Science Education and Research Bhopal, Bhopal Bypass Road, Bhauri, Bhopal 462066, Madhya Pradesh,
India}
\curraddr{}
\email{jpsaha@iiserb.ac.in}
\thanks{}



\begin{abstract}
Let $W$ be a nonempty subset of the set of integers $\mathbb Z$. A nonempty subset $C$ of $\mathbb Z$ is said to be an asymptotic complement to $W$ if $W+C$ contains almost all the integers except a set of finite size. The set $C$ is said to be a minimal asymptotic complement to $W$ if $C$ is an asymptotic complement to $W$, but $C\setminus \lbrace c\rbrace$ is not an asymptotic complement to $W$ for every $c\in C$. Asymptotic complements have been studied in the context of representations of integers since the time of Erd\H{o}s, Hanani, Lorentz and others, while the notion of minimal asymptotic complements is due to Nathanson. In this article, we study minimal asymptotic complements in $\mathbb{Z}$ and deal with a problem of Nathanson on their existence and their inexistence.
\end{abstract}

\maketitle

\section{Introduction}

\subsection{Background and Motivation}
Let $(G,\cdot)$ be a group where $\cdot$ denotes its group operation. If $A, B$ are two nonempty subsets of $G$, then we define the product set $A\cdot B$ as 
$$A\cdot B := \{a\cdot b \,|\, a\in A, b\in B\}.$$
If $A$ contains only one element $a$, then $A\cdot B$ is denoted by $a\cdot B$. 
Henceforth, we will omit the symbol ``$\,\cdot\,$'' and denote $A\cdot B$ by $AB$. If $A$ (resp. $B$) contains only one element $a$ (resp. $b$), then $AB$ is denoted by $aB$ (resp. $Ab$). 
Given two nonempty subsets $W,C$ of $G$, the set $C$ is called an asymptotic complement to $W$ if there exists a finite subset $F$ of $G$ such that $W C = G\setminus F$. The set $C$ is called a minimal asymptotic complement to $W$ if $C$ is an asymptotic complement to $W$, but $C\setminus \lbrace c\rbrace$ is not for every $c\in C$. In the context of abelian groups, the composition law is denoted by the symbol ``$+$" and  we will call the sets of the form $A+B$ as sumsets instead of product sets. 

Asymptotic complements have been studied since a long time in the context of representations of the integers. Indeed, one of the earliest instances of asymptotic complements can be found in the work of Lorentz \cite{Lorentz54} who showed that given an infinite set $W\subseteq \mathbb{N}$, there exists a set $C\subseteq \mathbb{N}$ of asymptotic density zero such that $W+C$ covers ``almost all" the positive integers\footnote{By ``almost all" we mean all except a set of finite size.}. This proved a conjecture of Erd\H{o}s (each infinite subset $W$ of $\bbN$ has an asymptotic complement of asymptotic density zero). Erd\H{o}s himself studied a number of questions related to the density of asymptotic complements for particular subsets $W$ of the positive integers. See \cite{Erdos54}, \cite{ErdosSomeUnsolved57} etc. Some more recent results on asymptotic complements can be found in \cite{WolkeOnProbErdosAddNoTh, FangChenAddCo1, FangChenAddCo2, FangChenAddCo3} etc.

The notion of minimal asymptotic complements was introduced by Nathanson in 2011 in the context of minimal bases (see \cite[\S 5]{NathansonAddNT4}). He also asked several questions related to the existence and inexistence of minimal asymptotic complements.

\begin{questionIntro}
\cite[Problem 14]{NathansonAddNT4}\label{nathansonprob14}
Let $W$ be a finite or infinite set of integers. 
\begin{enumerate}[(a)]
\item Does there exist a minimal asymptotic complement to $W$? 
\item Does each asymptotic complement to $W$ contain a minimal asymptotic complement? 
\end{enumerate}
\end{questionIntro}

\begin{questionIntro}
\cite[Problem 16]{NathansonAddNT4}\label{nathansonprob16}
Let $G$ be an infinite group, and let $W$ be a finite or infinite subset of $G$.
\begin{enumerate}[(a)]
\item Does there exist a minimal asymptotic complement to $W$? 
\item Does each asymptotic complement to $W$ contain a minimal asymptotic complement? 
\end{enumerate}
\end{questionIntro}

In the same article, Nathanson also introduced the concept of minimal complements. Given two nonempty subsets $W,C$ of a group $G$, the set $C$ is called a complement to $W$ if $W  C = G$. The set $C$ is called a minimal complement to $W$ if $C$ is a complement to $W$, but $C\setminus \lbrace c\rbrace$ is not for all $c \in C$. Henceforth, by a complement we shall mean as just mentioned above. If we mean set-theoretic complement, we shall explicitly state it. 

\subsection{Statements of results} 
In this article, we shall mainly concentrate on Question \ref{nathansonprob14} and derive sufficient conditions on the existence and the inexistence of minimal asymptotic complements in $\mathbb{Z}$ (see Theorems \ref{Thm.A}, \ref{Thm.B}, \ref{Thm.C} and \ref{Thm.D}). Moreover, some of our results hold in the full generality of arbitrary infinite groups. Thus, we deal partially with Question \ref{nathansonprob16} as well (see Theorems \ref{Thm.A}, \ref{Thm.B}).

First, we consider the finite subsets and subsets whose set-theoretic complement is finite.

\begin{theoremIntro}
\label{Thm.A}
Let $G$ be an infinite group and $W\subseteq G$.
\begin{enumerate}
\item If $W$ is finite and nonempty, then no asymptotic complement of $W$ contains a minimal asymptotic complement and $W$ does not admit any minimal asymptotic complement. 
\item If $G\setminus W$ is finite and nonempty\footnote{If $G\setminus W$ is empty, then any singleton subset of $G$ is a minimal complement and minimal asymptotic complement to $W$.}, then $W$ admits a minimal complement of cardinality two and the singleton subsets of $G$ are precisely the minimal asymptotic complements of $W$.
\end{enumerate}
\end{theoremIntro}

Theorem \ref{Thm.A}(1) answers each part of Questions \ref{nathansonprob14}, \ref{nathansonprob16} in negative when $W$ is finite, while Theorem \ref{Thm.A}(2) answers each part of Questions \ref{nathansonprob14}, \ref{nathansonprob16} in affirmative when $W$ has finite set-theoretic complement in $G$. The above result is proved in Section \ref{Sec:Struc}.

Next, we consider the subsets which are subgroups or translates of subgroups, and the subsets which contain subgroups.

\begin{theoremIntro}
\label{Thm.B}
Let $G$ be an infinite group.
\begin{enumerate}
\item Let $W$ be a (left or right) translate of an infinite subgroup of $G$. Then any asymptotic complement of $W$ contains a minimal asymptotic complement.
\item Suppose $G\setminus W$ is infinite and $W$ contains an infinite subgroup of $G$ of finite index in $G$. Then $W$ admits a minimal asymptotic complement and a minimal complement in $G$. In particular, the union of a nonzero subgroup of $\bbZ$ and a finite subset of $\bbZ$ admits a minimal asymptotic complement. 
\end{enumerate}
\end{theoremIntro}

Theorem \ref{Thm.B}(1) answers each part of Questions \ref{nathansonprob14}, \ref{nathansonprob16} in affirmative when $W$ is a translate of an infinite subgroup, while Theorem \ref{Thm.B}(2) answers Questions \ref{nathansonprob14}(a), \ref{nathansonprob16}(a) in affirmative when $W$ has infinite set-theoretic complement and contains an infinite subgroup of finite index. The above result is proved in Section \ref{Sec:Struc}. Moreover, we show that eventually periodic subsets (see Definition \ref{Def.2.4}) of $\bbZ$ do not admit any minimal asymptotic complement (see Lemma \ref{Lemma:EventuPeriNoMinAsymCom}).

\begin{remark}
The group $G$ in the statements of Theorems \ref{Thm.A}, \ref{Thm.B} is not assumed to be abelian.
\end{remark}

Finally, we concentrate on infinite subsets $W\subseteq \mathbb{Z}$ having less structure. In addition, we compare the seemingly close concepts of minimal complements and minimal asymptotic complements and show that even in the case of $\mathbb{Z}$ (which is a totally ordered abelian group), the existence of one does not automatically imply the existence of the other (see Lemmas \ref{Lemma:SufInfYesMac}, \ref{Lemma:SufInfNoMac}, \ref{Lemma:BddBelowYesMCNoMac}). In the following, $\bbZ_{?x}$ denotes the set $\{n\in \bbZ \,|\, n?x\}$ for $?\in \{<, \leq, >, \geq\}$ and $x\in \bbZ$. More generally, for a subset $A$ of $\bbZ$, $A_{?x}$ denotes the set $\{n\in A \,|\, n?x\}$ for $?\in \{<, \leq, >, \geq\}$ and $x\in \bbZ$. A nonempty subset $X$ of $\bbZ$ is said to be an interval if for any two elements $x_1, x_2\in X$ with $x_1\leq x_2$, the set $X$ contains $n$ for any integer $n$ satisfying $x_1\leq n\leq x_2$. The results below are proved in Section \ref{Sec:General}.

\begin{theoremIntro}
\label{Thm.C}
Let $I_1, J_1, I_2, J_2, \cdots$ be nonempty finite intervals in $\bbZ$ such that the following conditions hold. 
\begin{enumerate}[(i)]
\item For any positive integer $k$, 
$$\min J_k = \max I_k+1, \quad
\min I_{k+1} = \max J_k + 1.$$
\item The cardinalities $\# I_k$ of the sets $I_k$ form a strictly increasing sequence $\{\# I_k\}_{\geq 1}$.
\item The cardinalities $\# J_k$ of the sets $J_k$ form a strictly increasing sequence $\{\# J_k\}_{\geq 1}$.
\end{enumerate}
Then each of the following sets 
\begin{enumerate}
\item $\cup_{k=1}^\infty I_k$, 
\item $(\bbZ_{< \min I_1})\cup (\cup_{k=1}^\infty I_k)$, 
\item  $(\bbZ_{< \min I_1} \setminus F)\cup (\cup_{k=1}^\infty I_k)$ for any nonempty finite subset $F$ of $\bbZ$, 
\item 
$( \{x\in \bbZ_{< \min I_1} \,|\, x\equiv  a\modu n\})\cup (\cup_{k=1}^\infty I_k)$ 
for any two integers $a,n$ with $n\neq 0$,
\end{enumerate}
admits a minimal complement and no minimal asymptotic complement in $\bbZ$. 
\end{theoremIntro}

\begin{theoremIntro}
\label{Thm.D}
Let $W$ be a bounded below infinite subset of $\bbZ$ such that the limit of the difference of consecutive elements of $\bbZ_{\geq 1} \setminus W$ is equal to $\infty$. Then $W$ neither admits a minimal complement, nor admits a minimal asymptotic complement. 
\end{theoremIntro}

Note that if a subset $C$ of a group $G$ is a minimal asymptotic complement to a nonempty subset $W$ of $G$, then for any $g\in G$, the set $C$ (resp. $g^\mo C$) is a minimal asymptotic complement to $gW$ (resp. $Wg$). Thus any statement (for instance, the above results) about the existence of a minimal asymptotic complement of a particular subset $W$ of a group $G$ also remains valid for any of its left or right translates.

\section{Primer on Minimal complements}
In this section, we collect some known results on the existence of minimal complements in $\mathbb{Z}$ which will help us in comparing with minimal asymptotic complements in the subsequent sections. Nathanson showed that any nonempty finite subset of $\mathbb{Z}$ admits a minimal complement. In fact, he showed the following stronger result.

\begin{theorem}[{\cite[Theorem 8]{NathansonAddNT4}}]
\label{Thm.2.1}
Let $W$ be a nonempty, finite subset of the integers $\bbZ$. Every complement to $W$ in $\bbZ$ contains a minimal complement to $W$.
\end{theorem}

Chen--Yang were the first to give conditions on infinite sets admitting minimal complements and also not admitting minimal complements in $\mathbb{Z}$. They proved the following results.

\begin{theorem}[{\cite[Theorem 1]{ChenYang12}}]
\label{CY12Thm1}
Let $W$ be a set of integers such that $\inf W = -\infty$ and $\sup W = +\infty$. Then $W$ admits a minimal complement in $\bbZ$.
\end{theorem}

\begin{theorem}[{\cite[Theorem 2]{ChenYang12}}]
\label{CY12Thm2}
Let $W = \{1 = w_{1} < w_{2} < \cdots \}$ be a set of integers and
$\overline W = \bbZ_{\geq 1} \setminus W = \{ \bar w _1< \bar w_2 < \cdots \}$.
\begin{enumerate}[(a)]
\item If $\limsup_{i\rightarrow +\infty}(w_{i+1} - w_{i}) = +\infty$, then there exists a minimal complement to $W$.
\item  If $\lim_{i\rightarrow \infty}( \bar{w}_{i+1} - \bar{w}_{i}) = +\infty,$ then there does not exist a minimal complement to $W$. 
\end{enumerate}	
\end{theorem}
 
Later, Kiss--S\'andor--Yang \cite{KissSandorYangJCT19} introduced the notion of eventually periodic sets and gave necessary and sufficient conditions for them to have minimal complements or not.

\begin{definition}
\label{Def.2.4}
Let $A$ be a nonempty bounded below subset of the set of integers $\bbZ$. If there exists a positive integer $T$ such that $a + T \in A$ for all sufficiently large integers $a \in A$, then $A$ is called \textnormal{eventually periodic with period} $T$.
\end{definition}

However, we shall see in the following sections that minimal asymptotic complements might behave quite differently.

\section{Minimal asymptotic complement of structured sets}
\label{Sec:Struc}

In this section we prove Theorem \ref{Thm.A} and Theorem \ref{Thm.B}. We also show that no eventually periodic subset of $\mathbb{Z}$ admits a minimal asymptotic complement in $\mathbb{Z}$ (see Lemma \ref{Lemma:EventuPeriNoMinAsymCom}).

\begin{proof}
[Proof of Theorem \ref{Thm.A}]
Suppose $W$ is finite and nonempty. To prove Theorem \ref{Thm.A}(1), it is enough to show that for any finite subset $F$ of any asymptotic complement $C$ of $W$ in $G$, the set $C\setminus F$ is also an asymptotic complement of $W$ in $G$. Indeed, for such a finite set $F$, the set $WC$ is equal to $(W(C\setminus F)) \cup WF$. Since $G\setminus WC$ and $WF$ are finite, it follows that $C\setminus F$ is also an asymptotic complement of $W$ in $G$. This proves Theorem \ref{Thm.A}(1). 

Suppose $G\setminus W$ is finite and nonempty. Note that there exist elements $g_1, g_2, \cdots, g_n$ in $G$ such that $G\setminus W$ is equal to $\{g_1, \cdots, g_n\}$. Since $G$ is infinite and $G\setminus W$ is finite, it follows that $W$ is infinite. So $W$ contains an element $x$ which is not of the form $g_i g_j^\mo$ for some $i, j$. Thus for each $i$, the element $x^\mo g_i$ does not belong to $\{g_1, \cdots, g_n\} = G\setminus W$, i.e., $x^\mo g_i$ belongs to $W$. So $xW$ contains $g_1, g_2, \cdots, g_n$. Hence $\{e, x\}$ is a minimal complement to $W$.  
Since $G\setminus W$ is finite, it is follows that any singleton subset of $G$ is an asymptotic complement of $W$. 
\end{proof}

We contrast Theorem \ref{Thm.A}(1) with Theorem \ref{Thm.2.1} and \cite[Theorem A]{MinComp1} which ensures the existence of minimal complements of nonempty finite sets in arbitrary groups.

\begin{proof}
[Proof of Theorem \ref{Thm.B}]
First, we prove that any asymptotic complement of an infinite subgroup $H$ of $G$ contains a minimal asymptotic complement.
Let $H$ be an infinite subgroup of a group $G$. Let $C$ be an asymptotic complement of $H$ in $G$. Consider the equivalence relation $\sim$ on $C$ defined by: $c_1 \sim c_2$ if $c_1c_2^\mo \in H$. Let $C'$ denote a nonempty subset of $C$ consisting of pairwise inequivalent elements such that each element of $C$ is equivalent to some element of $C'$. Note that $HC$ is equal to $HC'$, and $HC$ is the union of a collection of right cosets of $H$. Since $G\setminus HC$ is finite and $H$ is infinite, it follows that $HC$ is equal to $G$ (this is clear from the fact that $HC$ is the union of certain right cosets of $H$ and $H$ is infinite). Consequently, $C$ is a complement of $H$ in $G$. Since $HC, HC'$ are equal, the set $C'$ is a complement of $H$ in $G$. Since the elements of $C'$ are inequivalent under $\sim$, it follows that $Hx, Hy$ are disjoint for any two elements $x, y \in C'$. The equality $G  = HC'$ implies that $G$ is equal to the union of the pairwise disjoint subsets of the form $Hx$ for $x$ varying in $C'$. Since $H$ is infinite, these subsets are all infinite. Thus $G\setminus (H (C'\setminus F))$ is infinite for any nonempty finite subset $F$ of $C'$. Consequently, the set $C'$ is a minimal asymptotic complement of $H$ in $G$. 

Let $g$ be an element of $G$ and $D$ be an asymptotic complement of $gH$ in $G$. Then $D$ is also an asymptotic complement of $H$. Hence $D$ is a complement of $H$ in $G$ and it contains a minimal asymptotic complement of $H$. Consequently, $D$ is a complement of $gH$ in $G$ and it contains a minimal asymptotic complement of $gH$.

If $E$ is an asymptotic complement of $Hg$ in $G$, then $E$ is an asymptotic complement of the subgroup $g^\mo Hg$ in $G$. Hence $E$ is a complement of $g^\mo Hg$ and it contains a minimal asymptotic complement of $g^\mo H g$. So $E$ is a complement of $Hg$ and it contains a minimal asymptotic complement of $H g$.

Theorem \ref{Thm.B}(1) follows from above. 

To prove Theorem \ref{Thm.B}(2), assume that $G \setminus W$ is infinite and $W$ contains a subgroup $H$ of $G$ having finite index in $G$. Since $H$ admits a finite subset of $G$ as a complement, it follows that $W$ also admits a finite subset $C$ of $G$ as a complement and as an asymptotic complement. 
Note that the subsets of $C$ which are asymptotic complements (resp. complements) to $W$ form a nonempty finite set $\calF$ (resp. $\calF'$) which is partially ordered with respect to containment and any minimal element of $\calF$ (resp. $\calF'$) is a minimal asymptotic complement (resp. minimal complement) of $W$ in $G$. Since each of the partially ordered sets $\calF, \calF'$ is finite and nonempty, they contain minimal elements. So $W$ admits a minimal complement and a minimal asymptotic complement in $G$. Consequently, the union of a nonzero subgroup of $\bbZ$ and a finite subset of $\bbZ$ admits a minimal complement. 
\end{proof}

\begin{lemma}
\label{Lemma:EventuPeriNoMinAsymCom}
Let $W$ be an eventually periodic subset of $\bbZ$ (see Definition \ref{Def.2.4}). Then $W$ admits no minimal asymptotic complement in $\bbZ$. 
\end{lemma}

\begin{proof}
On the contrary, let us assume that there exists a minimal asymptotic complement to $W$ in $\bbZ$. 
Let $C$ be such a minimal asymptotic complement. Since $W$ is bounded below, the set $C$ is infinite. Let $T$ denote a period of $W$. So $C$ contains at least two distinct elements $a, b$ which are congruent modulo $T$. Since $W$ is eventually periodic, for some finite subset $\scrW$ of $W$, it follows that $(b-a) + (W\setminus \scrW) \subseteq W$, which implies that $b + (W\setminus \scrW)$ is contained in $a + W$. Note that 
\begin{align*}
W + C 
& = ( W + b)  \cup ( W+ (C \setminus \{b\})) \\
& = (\scrW + b) \cup ((W\setminus \scrW) + b) \cup ( W+ (C \setminus \{b\})) \\
& \subseteq (\scrW + b) \cup (W + a) \cup ( W+ (C \setminus \{b\})) \\ 
& \subseteq (\scrW + b) \cup ( W+ (C \setminus \{b\})) ,
\end{align*}
where the final containment follows since $a$ lies in $C \setminus \{b\}$. This implies that $W + C  = (\scrW + b) \cup ( W+ (C \setminus \{b\}))$. Since $\scrW$ is finite, it follows that the set-theoretic complement of $W + (C \setminus \{b\})$ in $\bbZ$ is finite. Hence $C\setminus \{b\}$ is an asymptotic complement of $W$ in $\bbZ$, which contradicts the assumption. Consequently, $W$ admits no minimal asymptotic complement in $\bbZ$. 
\end{proof}

We contrast Lemma \ref{Lemma:EventuPeriNoMinAsymCom} with the work of  Kiss--S\'andor--Yang \cite[Theorems 2,3]{KissSandorYangJCT19} who showed that there exist eventually periodic sets in $\mathbb{Z}$ admitting minimal complements.

\section{Minimal asymptotic complements of general sets}
\label{Sec:General}

In this section, we consider infinite subsets $W\subseteq \mathbb{Z}$ which have ``less" structure. We shall see that the existence of minimal asymptotic complements is a trickier concept in this scenario. 

Let $W$ be a subset of $\bbZ$ such that $\sup W = +\infty$ and $\inf W= -\infty$. Then a minimal complement of $W$ exists by Theorem \ref{CY12Thm1}. However, a minimal asymptotic complement of $W$ may or may not exist (even if we impose the condition that $\bbZ\setminus W$ contains arbitrarily large gaps), as illustrated by Lemmas \ref{Lemma:SufInfYesMac}, \ref{Lemma:SufInfNoMac}.

\begin{lemma}
\label{Lemma:SufInfYesMac}
Let $W$ denote the subset of $\bbZ$ consisting of integers which are not primes. Then $W$ admits a minimal complement and a minimal asymptotic complement. 
\end{lemma}

\begin{proof}
Note that $W$ has $\{0, 1, -1\}$ is a minimal complement of $W$ and $\{0,1\}$ is a minimal asymptotic complement of $W$. 
\end{proof}

\begin{remark}
The above Lemma can also be seen as a consequence of Theorem \ref{Thm.B}(2) (since $W$ contains the subgroup $4\bbZ$ and $\bbZ\setminus W$ is infinite). 
\end{remark}

On the other hand, the next lemma gives an example of a subset $W$ of $\bbZ$ with $\sup W = +\infty$ and $\inf W= -\infty$ and which admits no minimal asymptotic complement. 
 
\begin{lemma}
\label{Lemma:SufInfNoMac}
Let $W$ denote the subset of $\bbZ$ obtained by taking the union of $\bbZ_{\leq 3}$ and 
\begin{align*}
& \bigcup _{k=1}^\infty \left( ((1+2)+(2+2^2 ) + (3 + 2^3) + \cdots + (k + 2^k) ) +  [1, k+1] \right) \\
& =
\bigcup _{k=1}^\infty   \left[\dfrac{k(k+1)}{2} + 2^{k+1}-2+1, \dfrac{k(k+1)}{2} + 2^{k+1}-2+k+1 \right]  \\
& =
\bigcup _{k=1}^\infty  \left[\dfrac{(k-1)(k+2)}{2} + 2^{k+1}, \dfrac{(k-1)(k+2)}{2} + 2^{k+1} + k\right] \\
& = \{4, 5\} \cup \{10, 11, 12\} \cup \{21, 22, 23, 24\}\cup \{41, 42, 43, 44, 45\} \cup \cdots .
\end{align*}
The set $W$ admits a minimal complement in $\bbZ$. For any three elements $a, b, c$ in an asymptotic complement $C$ of $W$ with $a<b<c$, the set $C\setminus \{b\}$ is also an asymptotic complement to $W$. Moreover, $W$ admits no minimal asymptotic complement in $\bbZ$. 
\end{lemma}

\begin{proof}
The set $W$ admits a minimal complement in $\bbZ$ by Theorem \ref{CY12Thm1}. 

Note that for any two integers $u,v$ with $0 \leq u\leq v$, it follows that 
$$u+  \{x, x+1, \cdots, y\} 
\subseteq 
\{0,v\} + \{x, x+1, \cdots, y\}
$$
for any two integers $x,y$ with $y-x\geq v$.
Indeed, the set 
$$\{0,v\} + \{x, x+1, \cdots, y\}
= \{x, x+1, \cdots, y\} \cup \{v + x, v+ x+1, \cdots, v+y\}
$$
consists of the integers $k$ satisfying $x \leq k \leq v + y$ (since $v+x \leq y$) and the elements of the set 
$u+  \{x, x+1, \cdots, y\}$ lie between $x$ and $v+y$ (since $u + x \geq x, u + y \leq v+y$). 
Since $a<b<c$ are integers, 
\begin{align*}
&(b-a)+  \left[\dfrac{(k-1)(k+2)}{2} + 2^{k+1}, \dfrac{(k-1)(k+2)}{2} + 2^{k+1} + k\right]  \\
&\subseteq 
\{0,c-a\} + 
 \left[\dfrac{(k-1)(k+2)}{2} + 2^{k+1}, \dfrac{(k-1)(k+2)}{2} + 2^{k+1} + k\right] 
\end{align*}
holds for any integer $k >  c-a$. Consequently, for any integer $k>c-a$, we obtain 
\begin{align*}
&b+  \left[\dfrac{(k-1)(k+2)}{2} + 2^{k+1}, \dfrac{(k-1)(k+2)}{2} + 2^{k+1} + k\right]  \\
&\subseteq 
\{a,c\} + 
 \left[\dfrac{(k-1)(k+2)}{2} + 2^{k+1}, \dfrac{(k-1)(k+2)}{2} + 2^{k+1} + k\right] .
\end{align*}
Also note that for any $w\in W$ with $w\leq -(b-a)$, we have 
$$w + b-c \leq - (b-a) +b-c = a-c < 0,$$ 
which implies $w + b-c\in W$ (since $W$ contains $\bbZ_{\leq 3}$), and hence 
$$b+ w 
= c + (w + b-c) \in c+ W \subseteq \{a, c\} + W.$$
Consequently,
\begin{align*}
& b + \left(W \setminus \left[-(b-a), \dfrac{(c-a-1)(c-a+2)}{2} + 2^{c-a+1} + c-a \right] \right) \\
& \subseteq 
\{a,c\}+ W \\
& \subseteq 
(C\setminus \{b\} ) + W.
\end{align*} 
This implies that 
\begin{align*}
C + W 
& = 
((C\setminus \{b\} ) + W) 
\cup 
(b + W) \\
& = 
((C\setminus \{b\} ) + W)  \\
& \qquad \cup \left(b + \left(W \setminus \left[-(b-a), \dfrac{(c-a-1)(c-a+2)}{2} + 2^{c-a+1} + c-a\right] \right) \right)  \\
& \qquad \cup \left(b +  \left(W\cap \left[-(b-a), \dfrac{(c-a-1)(c-a+2)}{2} + 2^{c-a+1} + c-a\right]\right) \right) \\
& = 
((C\setminus \{b\} ) + W) \\
& \qquad \cup \left(b +  \left(W\cap \left[-(b-a), \dfrac{(c-a-1)(c-a+2)}{2} + 2^{c-a+1} + c-a\right]\right) \right).
\end{align*}
Thus the set-theoretic complement of $(C\setminus \{b\} ) + W$ in $C+W$ is a finite set. Since the set-theoretic complement of $C+W$ in $\bbZ$ is finite, it follows that the set-theoretic complement of $(C\setminus \{b\} ) + W$ in $\bbZ$ is also finite. So $C\setminus\{b\}$ is also an asymptotic complement to $W$. 

We claim that each asymptotic complement of $W$ contains at least three distinct elements. Since $\bbZ \setminus W$ is infinite, no singleton subset of $\bbZ$ is an asymptotic complement to $W$. If a two-element subset $S = \{s, t\}$ of $\bbZ$ with $s<t$ is an asymptotic complement to $W$, then $(-s) + S$ is also an asymptotic complement to $W$. Hence we may assume $S$ is equal to $\{0,u\}$ with $u>0$. Let $k_0\geq 1$ be an integer such that $2^{k_0+ 1}> u$. 
For every integer $k\geq 1$, the integer 
$\frac{(k-1)(k+2)}{2} + 2^{k+1} + k + 1 + u$
does not belong to $W+u$. 
Moreover, for any integer $k\geq k_0$, the integer 
$\frac{(k-1)(k+2)}{2} + 2^{k+1} + k + 1 + u$
does not belong to $W$. 
Otherwise, for some integer $k'\geq k_0$, the integer 
$\frac{(k'-1)(k'+2)}{2} + 2^{k'+1} + k' + 1 + u
$
belongs to 
$$
\bigcup _{k> k'}^\infty  \left[\dfrac{(k-1)(k+2)}{2} + 2^{k+1}, \dfrac{(k-1)(k+2)}{2} + 2^{k+1} + k\right],
$$
which implies 
$$
\frac{(k'-1)(k'+2)}{2} + 2^{k'+1} + k' + 1 + u
\geq 
\dfrac{(k'+1-1)(k'+1+2)}{2} + 2^{k'+1+1} 
,$$
which yields $u \geq 2^{k'+1} \geq 2^{k_0+1}$. This contradicts the inequality $2^{k_0+ 1}> u$. So for any integer $k\geq k_0$, the integer 
$\frac{(k-1)(k+2)}{2} + 2^{k+1} + k + 1 + u$
does not belong to $\{0,u\} + W = S + W$. 
Hence $S$ is not a minimal asymptotic complement to $W$. This proves the claim that each asymptotic complement of $W$ contains at least three elements. Hence $W$ admits no minimal asymptotic complement in $\bbZ$.
\end{proof}

The natural question now is what happens if the set $W$ is bounded from above or from below. In particular, let $W$ satisfy the condition of Theorem \ref{CY12Thm2}(a) (i.e. the difference of consecutive elements of $W$ is not bounded above). We shall see that even in this case minimal asymptotic complement might not exist.

\begin{lemma}
\label{Lemma:BddBelowYesMCNoMac}
Let $W$ denote the subset
\begin{align*}
& \bigcup _{k=1}^\infty \left( ((1+2)+(2+2^2 ) + (3 + 2^3) + \cdots + (k + 2^k)) +  [1, k+1] \right) \\
& =
\bigcup _{k=1}^\infty   \left[\dfrac{k(k+1)}{2} + 2^{k+1}-2+1, \dfrac{k(k+1)}{2} + 2^{k+1}-2+k+1 \right]  \\
& =
\bigcup _{k=1}^\infty  \left[\dfrac{(k-1)(k+2)}{2} + 2^{k+1}, \dfrac{(k-1)(k+2)}{2} + 2^{k+1} + k\right] \\
& = \{4, 5\} \cup \{10, 11, 12\} \cup \{21, 22, 23, 24\}\cup \{41, 42, 43, 44, 45\} \cup \cdots 
\end{align*}
of $\bbZ$. 
The set $W$ admits a minimal complement in $\bbZ$. For any three elements $a, b, c$ in an asymptotic complement $C$ of $W$ with $a<b<c$, the set $C\setminus \{b\}$ is also an asymptotic complement to $W$. Moreover, $W$ admits no minimal asymptotic complement in $\bbZ$. 
\end{lemma}

\begin{proof}
The set $W$ admits a minimal complement in $\bbZ$ by Theorem \ref{CY12Thm2}(a). 

Let $a, b, c$ be elements of an asymptotic complement $C$ of $W$ with $a<b<c$. Following the same argument as in the proof of Lemma \ref{Lemma:SufInfNoMac}, it follows that 
$$
b+  \left[\dfrac{(k-1)(k+2)}{2} + 2^{k+1}, \dfrac{(k-1)(k+2)}{2} + 2^{k+1} + k\right] $$
is contained in 
$$ 
\{a,c\} + 
 \left[\dfrac{(k-1)(k+2)}{2} + 2^{k+1}, \dfrac{(k-1)(k+2)}{2} + 2^{k+1} + k\right] $$
for any $k>c-a$.
This implies 
\begin{align*}
& b + \left(W \setminus \left[4, \dfrac{(c-a-1)(c-a+2)}{2} + 2^{c-a+1} + c-a\right] \right) \\
& \subseteq 
\{a,c\}+ W \\
& \subseteq 
(C\setminus \{b\} ) + W.
\end{align*} 
Using a similar argument as in the proof of Lemma \ref{Lemma:SufInfNoMac}, it follows that 
\begin{align*}
C + W 
& = ((C\setminus \{b\} ) + W) \\
& \qquad \cup \left(b +  \left(W\cap \left[4, \dfrac{(c-a-1)(c-a+2)}{2} + 2^{c-a+1} + c-a\right] \right)\right),
\end{align*}
and hence $C\setminus \{b\}$ is an asymptotic complement to $W$. 

To prove that $W$ does not admit a minimal asymptotic complement, it suffices to prove that any asymptotic complement of $W$ contains at least three elements, which can be proved using a similar argument as in the proof of Lemma \ref{Lemma:SufInfNoMac}. 
\end{proof}

\begin{remark}
Similarly, it can be proved that there are subsets of $\bbZ$ bounded from above which admit minimal complements, but admit no minimal asymptotic complements. For instance, we could consider $-W = \{-w\,|\,w\in W\}$ where $W$ is as in Lemma \ref{Lemma:BddBelowYesMCNoMac} and use the fact that multiplication by $-1$ is an automorphism of the group $\bbZ$.  
\end{remark}

We shall now prove Theorem \ref{Thm.C}, which is a general result about the existence of minimal complements and the inexistence of minimal asymptotic complements.

\begin{proof}[Proof of Thorem \ref{Thm.C}]
We refer to the sets as in part (1), (2), (3), (4) of Theorem \ref{Thm.C} as the first, second, third, fourth set respectively.
The first set, i.e., $\cup_{k=1}^\infty I_k$ admits a minimal complement in $\bbZ$ by Theorem \ref{CY12Thm2}(a). The second, third and the fourth sets also admit minimal complements in $\bbZ$ by Theorem \ref{CY12Thm1}.

Let $a, b, c$ be three elements in $\bbZ$ with $a<b<c$. From condition (ii), it follows that $\# I_k\geq k$ for any $k\geq 1$. So for any $k\geq c-a$, it follows that $(b-a) + I_k \subseteq \{0, c\} + I_k$, which gives $b + I_k \subseteq \{a, c\} + I_k$, and hence 
\begin{equation}
\label{Eqn:abc}
b + \cup_{k=c-a}^\infty I_k \subseteq \{a, c\} + \cup_{k=c-a}^\infty I_k \subseteq \{a, c\} + \cup_{k=1}^\infty I_k.
\end{equation}
Note that for integers $x, u, v$ with $u<v$, 
\begin{equation}
\label{Eqn:lambdauv}
u + \bbZ_{\leq x } \subseteq v + \bbZ_{\leq x }.
\end{equation}
Also note that by condition (iii), any minimal complement of any one of the first, second, third and the fourth set is infinite. 

For any three distinct elements $a, b, c$ in an asymptotic complement $C$ of $\cup_{k=1}^\infty I_k$ with $a<b<c$, we obtain from Equation \eqref{Eqn:abc} that 
$$
b + \cup_{k=c-a}^\infty I_k \subseteq \{a, c\} + \cup_{k=1}^\infty I_k
 \subseteq (C\setminus \{b\}) + \cup_{k=1}^\infty I_k
.$$
Consequently, $\cup_{k=1}^\infty I_k$ does not admit any minimal asymptotic complement.

For any three distinct elements $a, b, c$ in an asymptotic complement $C$ of the second set with $a<b<c$, we obtain from Equations \eqref{Eqn:abc}, \eqref{Eqn:lambdauv} that 
$$
b + \cup_{k=c-a}^\infty I_k \subseteq \{a, c\} + \cup_{k=1}^\infty I_k
 \subseteq (C\setminus \{b\}) + \cup_{k=1}^\infty I_k,
$$
$$
b + \bbZ_{< \min I_1} 
\subseteq 
c + \bbZ_{< \min I_1} 
 \subseteq (C\setminus \{b\}) +  \bbZ_{< \min I_1} 
.$$
So 
$$b + \left((\bbZ_{< \min I_1})\cup (\cup_{k=c-a}^\infty I_k) \right) 
\subseteq 
\{a, c\} + 
\left( (\bbZ_{< \min I_1})\cup (\cup_{k=1}^\infty I_k) \right)
.$$
Consequently, $(\bbZ_{< \min I_1})\cup (\cup_{k=1}^\infty I_k)$ does not admit any minimal asymptotic complement.

For any three distinct elements $a, b, c$ in an asymptotic complement $C$ of the third set with $a<b<c$, we obtain from Equations \eqref{Eqn:abc}, \eqref{Eqn:lambdauv} that 
$$
b + \cup_{k=c-a}^\infty I_k \subseteq \{a, c\} + \cup_{k=1}^\infty I_k
 \subseteq (C\setminus \{b\}) + \cup_{k=1}^\infty I_k,
$$
$$
b + \bbZ_{< \min \{\min F, \min I_1\}} 
\subseteq 
c + \bbZ_{< \min \{\min F, \min I_1\}} 
 \subseteq (C\setminus \{b\}) + ( \bbZ_{< \min I_1} \setminus F)
.$$
So 
$$b + \left((\bbZ_{< \min \{\min F, \min I_1\}})\cup (\cup_{k=c-a}^\infty I_k) \right) 
\subseteq 
\{a, c\} + 
\left( (\bbZ_{< \min I_1} \setminus F )\cup (\cup_{k=1}^\infty I_k) \right)
.$$
Consequently, $(\bbZ_{< \min I_1}\setminus F)\cup (\cup_{k=1}^\infty I_k)$ does not admit any minimal asymptotic complement.

Since any asymptotic complement $C$ to the fourth set is infinite, it contains three distinct elements $a<b<c$ which are congruent modulo $n$. From Equation \eqref{Eqn:abc} it follows that 
$$
b + \cup_{k=c-a}^\infty I_k \subseteq \{a, c\} + \cup_{k=1}^\infty I_k
 \subseteq (C\setminus \{b\}) + \cup_{k=1}^\infty I_k.
$$
Since $b<c$ and $b$ is congruent to $c$ modulo $n$, we obtain 
\begin{align*}
b +  \{x\in \bbZ_{< \min I_1} \,|\, x\equiv  a\modu n\}
& \subseteq 
c +  \{x\in \bbZ_{< \min I_1} \,|\, x\equiv  a\modu n\}\\
&
 \subseteq (C\setminus \{b\}) +  \{x\in \bbZ_{< \min I_1} \,|\, x\equiv  a\modu n\}.
\end{align*}
So 
\begin{align*}
&b + \left(( \{x\in \bbZ_{< \min I_1} \,|\, x\equiv  a\modu n\})\cup (\cup_{k=c-a}^\infty I_k) \right) \\
&\subseteq 
\{a, c\} + 
\left( ( \{x\in \bbZ_{< \min I_1} \,|\, x\equiv  a\modu n\})\cup (\cup_{k=1}^\infty I_k) \right).
\end{align*}
Consequently, $( \{x\in \bbZ_{< \min I_1} \,|\, x\equiv  a\modu n\})\cup (\cup_{k=1}^\infty I_k)$ does not admit any minimal asymptotic complement.
\end{proof}

Lemma \ref{Lemma:BddBelowYesMCNoMac} and Theorem \ref{Thm.C}(1) give examples of bounded below subsets of $\bbZ$, and for each of them, the difference of consecutive elements is not bounded above, i.e., the hypothesis of Theorem \ref{CY12Thm2}(a) holds. Now consider the situation where the hypothesis of Theorem \ref{CY12Thm2}(a) does not hold, i.e., the difference of consecutive elements is bounded. Such a bounded below subset of $\bbZ$ may or may not admit a minimal complement, as mentioned in \cite[Theorem 4, Remark 2]{KissSandorYangJCT19}. Theorem \ref{Thm.D} considers such subsets which do not admit any minimal complement. Now let us show the Theorem.

\begin{proof}
[Proof of Theorem \ref{Thm.D}]
The set $W$ admits no minimal complement in $\bbZ$ by Theorem \ref{CY12Thm2}(b).

Let $w_1<w_2<w_3<\cdots$ denote the elements of $W$, and $v_1 < v_2 < v_3 <\cdots $ denote the elements of $\bbZ_{\geq w_1}\setminus W$. Let $i'$ be a positive integer such that $v_{i+1} - v_i\geq 2$ for all $i\geq i'$. Let $C$ be an asymptotic complement of $W$ in $\bbZ$. Since $W$ is bounded below, the set $C$ is infinite. Let $a, b, c$ be three elements of $C$ such that $a<b<c$. Let $i''$ be a positive integer such that $i''\geq i'$ and $v_{i+1} - v_i\geq c-a$ for all $i\geq i''$. For any $i\geq i''$, 
we obtain 
$$
b-a 
+ 
(v_i, v_{i+1}) 
\subseteq 
(v_i, v_{i+1}) \cup ((c-a) + (v_{i}, v_{i+1}))
\subseteq 
\{0, c-a\} + (v_i, v_{i+1}) ,
$$
which gives 
$$
b + 
(v_i, v_{i+1}) \subseteq 
\{a, c\} + (v_i, v_{i+1}) 
 \subseteq 
\{a, c\} + W .
$$
Since $i''\geq i'$, it follows that $v_{i+1} - v_i\geq 2$ for all $i\geq i''$ and hence $W_{\geq v_{i''}}$ is equal to $(v_{i''}, v_{i''+1}) \cup (v_{i''+1}, v_{i''+2}) \cup (v_{i''+2}, v_{i''+3}) \cup (v_{i''+3}, v_{i''+4})\cup \cdots$. So the set $b + W_{\geq v_{i''}}$ is contained in $\{a, c\} + W$. 
Since $W$ is bounded below, it follows that $C\setminus \{b\}$ is also an asymptotic complement of $W$. Hence $W$ admits no minimal asymptotic complement. 
\end{proof}

We conclude with the following corollary.

\begin{corollary}
Let $W$ denote the set of all positive integers which do not belong to 
$$\bigcup _{k=1}^\infty [10k^2, 10k(k+1)].$$
The set $W$ does not admit a minimal complement in $\bbZ$ and $W$ admits no minimal asymptotic complement in $\bbZ$. Moreover, for any three elements $a, b, c$ in an asymptotic complement $C$ of $W$ with $a<b<c$, the set $C\setminus \{b\}$ is also an asymptotic complement to $W$. 
\end{corollary}

\begin{proof}
Note that the set $W$ as above satisfies the hypothesis of Theorem \ref{Thm.D}. Hence the first part follows. The second part follows from the proof of Theorem \ref{Thm.D}. 
\end{proof}

\vspace*{1mm}

\section{Acknowledgements}
We wish to thank the anonymous reviewer for the valuable comments and suggestions.
The first author would like to acknowledge the fellowship of the Erwin Schr\"odinger International Institute for Mathematics and Physics (ESI) and would also like to thank the Fakult\"at f\"ur Mathematik, Universit\"at Wien where a part of the work was carried out. The second author would like to acknowledge the Initiation Grant from the Indian Institute of Science Education and Research Bhopal, and the INSPIRE Faculty Award from the Department of Science and Technology, Government of India.
\def\cprime{$'$} \def\Dbar{\leavevmode\lower.6ex\hbox to 0pt{\hskip-.23ex
  \accent"16\hss}D} \def\cfac#1{\ifmmode\setbox7\hbox{$\accent"5E#1$}\else
  \setbox7\hbox{\accent"5E#1}\penalty 10000\relax\fi\raise 1\ht7
  \hbox{\lower1.15ex\hbox to 1\wd7{\hss\accent"13\hss}}\penalty 10000
  \hskip-1\wd7\penalty 10000\box7}
  \def\cftil#1{\ifmmode\setbox7\hbox{$\accent"5E#1$}\else
  \setbox7\hbox{\accent"5E#1}\penalty 10000\relax\fi\raise 1\ht7
  \hbox{\lower1.15ex\hbox to 1\wd7{\hss\accent"7E\hss}}\penalty 10000
  \hskip-1\wd7\penalty 10000\box7}
  \def\polhk#1{\setbox0=\hbox{#1}{\ooalign{\hidewidth
  \lower1.5ex\hbox{`}\hidewidth\crcr\unhbox0}}}
\providecommand{\bysame}{\leavevmode\hbox to3em{\hrulefill}\thinspace}
\providecommand{\MR}{\relax\ifhmode\unskip\space\fi MR }
\providecommand{\MRhref}[2]{%
  \href{http://www.ams.org/mathscinet-getitem?mr=#1}{#2}
}
\providecommand{\href}[2]{#2}

\end{document}